\documentclass[a4paper,preprint,review]{article}

\usepackage{amsthm}
\usepackage{amsfonts}
\usepackage{amsmath}
\usepackage{amssymb,enumerate}
\usepackage{graphicx}
\usepackage[T1]{fontenc}
\newtheorem{theorem}{Theorem}
\newtheorem{lemma}[theorem]{Lemma}
\newtheorem{corollary}[theorem]{Corollary}
\newtheorem{proposition}[theorem]{Proposition}

\newtheorem{observation}[theorem]{Observation}

\theoremstyle{definition}

\newtheorem{problem}[theorem]{Problem}

\theoremstyle{remark}

\newcommand{\cH}{{\cal H}}
\newcommand{\cF}{{\cal F}}
\newcommand{\cS}{{\cal S}}
\newcommand{\cX}{{\cal X}}
\newcommand{\cC}{{\cal C}}
\newcommand{\cK}{{\cal K}}
\newcommand{\cP}{{\cal P}}
\newcommand{\cA}{{\cal A}}
\newcommand{\cB}{{\cal B}}
\newcommand{\cT}{{\cal T}}

\DeclareMathOperator{\diam}{diam}
\DeclareMathOperator{\dist}{dist}
\DeclareMathOperator{\mut}{\mu_{\rm t}}
\DeclareMathOperator{\mud}{\mu_{\rm d}}
\DeclareMathOperator{\muo}{\mu_{\rm o}}
\DeclareMathOperator{\gp}{gp}
\DeclareMathOperator{\ex}{ex}

\usepackage{xcolor}
\usepackage[normalem]{ulem}

\begin{document}

%\begin{frontmatter}


\title{Mutual-visibility problems in Kneser and Johnson graphs}

\author{G\" ulnaz Boruzanl{\i} Ekinci$^{a}$ \footnote{Corresponding author.} \qquad \qquad  Csilla Bujt\'{a}s$^{b,c}$ \\\\
	$^{a}$ \small Department of Mathematics, Ege University,
	Izmir, Turkey\\
	\small{\tt  gulnaz.boruzanli@ege.edu.tr}\\
		$^{b}$ \small Institute of Mathematics, Physics and Mechanics, Ljubljana, Slovenia \\
	$^{c}$ \small University of Pannonia, Veszpr\' em, Hungary\\
		\small {\tt bujtasc@gmail.com}\\
}
%\date{}

\maketitle

\begin{abstract}
	Let $G$ be a connected graph and $\cX \subseteq V(G)$. By definition, two vertices $u$ and $v$ are $\cX$-visible in $G$ if there exists a shortest $u,v$-path with all internal vertices being outside of the set $\cX$. The largest size of $\cX$ such that any two vertices of $G$ (resp.\ any two vertices from $\cX$) are $\cX$-visible is the total mutual-visibility number (resp.\ the mutual-visibility number) of $G$.
	
In this paper, we determine the total mutual-visibility number of Kneser graphs, bipartite Kneser graphs, and Johnson graphs. The formulas proved for Kneser, and bipartite Kneser graphs are related to the size of transversal-critical uniform hypergraphs, while the total mutual-visibility number of Johnson graphs is equal to a hypergraph Tur\' an number. Exact values or estimations for the mutual-visibility number over these graph classes are also established.
    
\end{abstract}

\medskip\noindent
\textbf{AMS Math.\ Subj.\ Class.\ (2020)}: 05C12, 05C38, 05C65, 05C76
%05C40 \sep 94C15 \sep 05D05

%\linenumbers

\section{Introduction} \label{Sect-Intro}

% For a set of vertices of $KG(n,k)$ we'll use $\cS$
 %underlying k-uniform hypergraph $\cF(\cS) $ for a set  $\cS \subseteq V(KG(n,k))$

 %$\cX$
 
% By definition every total mutual visibility set is a mutual visibility set and then $\mu(G) \geq \mut(G)$ for every graph $G$. Transversal number $\tau(\cH)$
 %$k\geq 2$
 %\begin{theorem}\label{Thm-diam}
 %	The diameter 
 %\end{theorem}

Several studies have been conducted to address network visibility problems, both from theoretical and practical points of view. From the practical perspective, many papers focus on robot navigation problems (see, e.g., \cite{adhikary-2022, Cicerone-2023+, diluna-2017, poudel-2021}). From the theoretical perspective, the mutual-visibility notion in graphs was proposed by Di Stefano \cite{distefano-2022} and then investigated in several graph classes \cite{Cicerone-2023+3, Cicerone-2023+5, Cicerone-2023, Cicerone-2023+2, korze-2023}. As a natural extension, Cicerone \textit{et al.}\ \cite{Cicerone-2023+} proposed the notion of total mutual-visibility to provide a ``visibility'' property to all the network nodes instead of the nodes where robots are located only. Total mutual-visibility sets have been recently investigated in product graphs \cite{kuziak-2023+, tian-2023+} and Hamming graphs \cite{Bujtas-2023}. It is also shown that computing both of the mutual-visibility number and the total mutual-visibility number are NP-complete in general \cite{Cicerone-2023+4,distefano-2022}. Very recently, some new variants of the mutual-visibility notion were defined and discussed for different graph classes \cite{bresar2023lower, Cicerone-2023+4, Cicerone-2024+}. 
\medskip

For a connected graph $G$ with vertex set $V(G)$ and edge set $E(G)$, let $\cX$ be a subset of $V(G)$. Two vertices $u$ and $v$ from $V(G)$ are \textit{$\cX$-visible} if there exists a shortest $u,v$-path such that no internal vertex belongs to $\cX$. That is, $u$ and $v$ are $\cX$-visible if  there exists a shortest $u,v$-path $P$ in $G$ such that $ V(P) \cap \cX \subseteq \{u,v\}$.  If any two vertices from $\cX$  are $\cX$-visible then the set $\cX$ is a \textit{mutual-visibility set}, while if any two vertices from $V(G)$ are $\cX$-visible, then the set $\cX$ is a \textit{total mutual-visibility set} in $G$. The cardinality of a largest mutual-visibility (resp. total mutual-visibility) set in $G$ is defined to be the \textit{mutual-visibility number} (resp. \textit{total mutual-visibility number}), $\mu(G)$ (resp. $\mut(G)$), of $G$. By definition, $\mu(G) \geq \mut(G)$ holds for every (connected) graph $G$.
%The set of all $k$--subsets of $[n]$ is denoted by ${{[n]}\choose k}$. 
\medskip

It has been proved that the mutual-visibility problems are connected to several classical combinatorial problems which are extremely difficult to solve. To determine the mutual-visibility numbers for Cartesian products of two complete graphs is equivalent to the famous problem of Zarankievicz~\cite{Cicerone-2023}. The total mutual-visibility number of Hamming graphs and line graphs of complete graphs can be expressed as Tur\'an numbers for graphs or hypergraphs~\cite{Bujtas-2023, Cicerone-2024+}.

\medskip

In this manuscript, we prove exact formulas for the total mutual-visibility numbers of Kneser graphs, bipartite Kneser graphs, and Johnson graphs. For Kneser and bipartite Kneser graphs, these formulas show a direct connection to the minimum size of covering designs. For Johnson graphs, we prove that  $\mut(G)$ equals a hypergraph Tur\'an number.
Along the way, we also obtain exact values or estimations for the mutual-visibility number over these graph classes.

%%%%%%%%%%%%%%%%%%%%%%%%%%%%%%%%%%%%%%%%%%%%%%%%%
\subsection{Terminology}

This paper considers only simple undirected graphs that are \textit{connected}. The \textit{distance} between two vertices $u,v \in V(G)$, denoted by $\dist(u,v)$, is the length of a shortest path between $u$ and $v$. The maximum distance between any two vertices in $G$ is the \textit{diameter} of $G$ and is denoted by $\diam(G)$. Throughout the paper, $[n]$ stands for the set $\{1, \dots, n\}$ for every positive integer $n$. For graph theory concepts not defined here, we refer the reader to the book~\cite{west-2021}.
\medskip

A \textit{hypergraph} $\cH$ is a set system over the underlying vertex set $V(\cH)$. We assume, as it is usual, that the edge set $E(\cH)$ does not contain the empty set that is $E(\cH) \subseteq 2^{V(\cH)} \setminus \{\emptyset\}$. A hypergraph is \textit{$k$-uniform}, if every (hyper)edge $e \in E(\cH)$ consists of exactly $k$ vertices. Then, $2$-uniform hypergraphs correspond to graphs without loops. A $k$-uniform hypergraph is often called \textit{$k$-graph}, while an edge $e \in E(\cH)$ is a \textit{$k$-edge} if $|e|=k$. The \textit{degree} of a vertex $v \in V(\cH)$ is the number of edges incident to $v$, and $v$ is called \textit{isolated vertex} (isolate) if its degree is $0$. For a set $Y \subseteq V(\cH)$, the \textit{induced subhypergraph} contains those edges of $\cH$ that are fully contained in $Y$.

A set $T \subseteq V(\cH)$ is a \textit{transversal} in $\cH$ if $e\cap T \neq \emptyset$ for every edge $e \in E(\cH)$. The minimum cardinality of a transversal is the \textit{transversal number} $\tau(\cH)$ of the hypergraph.\footnote{Transversals are often called vertex cover, especially related to graphs. However, in this paper, we use the term ``cover'' in a different meaning (for subset relation between sets) and keep the term ``transversal'' as defined here.}
\medskip

Let $n$ and $k$ be integers such that $n\geq k\geq 1$. The \textit{Kneser graph} $KG(n,k)$ is the graph with vertices representing the $k$-subsets of the \textit{ground set} $[n]$, and where two vertices $A$ and $B$ are adjacent if and only if $A\cap B=\emptyset$ holds for the corresponding sets. Formally, the vertex set is $V(KG(n,k))={{[n]}\choose k}$ and the edge set is $$E(KG(n,k))=\{\{A,B\} : A,B\in V(KG(n,k)) {\rm ~and~} A\cap B=\emptyset\}.$$ 
This family of graphs was introduced in 1955 by M.~Kneser.  If $n<2k$, then $KG(n,k)$ consists of $n \choose k$ isolated vertices, and if $n=2k$, then $KG(n,k)$ is the union of disjoint copies of $K_2$. Moreover, $KG(n,1)$ is the complete graph on $n$ vertices. In this paper, we consider only connected Kneser graphs.  Thus, to avoid some trivial cases, we assume that $n\geq 2k+1$ and $k \ge 2$ for a Kneser graph $KG(n,k)$. 

The vertex set of a \textit{bipartite Kneser graph} $H(n,k)$ consists of vertices representing the $k$-element and $(n-k)$-element subsets of the ground set $[n]$. Two vertices $u$ and $v$ are adjacent if one of the represented sets is a subset of the other. By definition, $H(n,k)$ is bipartite for every $n >2k \ge 2$. The partite classes correspond to  ${[n] \choose k}$ and ${[n] \choose n-k}$, respectively. As $H(n,k) \cong H(n,n-k)$ and $H(2k,k)$ consists of $k$ copies of $K_2$-components, we will assume $n \ge 2k+1 \ge 5$ while studying mutual-visibility problems in bipartite Kneser graphs. 

The \emph{Johnson graph} $J(n,k)$ is the graph with vertex set consisting of all the $k$-subsets of the ground set $[n]$ and the edge set is 
$$E(J(n,k))=\{\{A,B\} : A,B\in V(J(n,k)) {\rm ~and~} |A\cap B|=k-1\}.$$
Observe that, by definition, $J(n,2)$ corresponds to the line graph of the complete graph $K_n$.
It is well-known that  $J(n,k) $ is isomorphic to $J(n,n-k)$ and that  $J(n,1)$ is isomorphic to the complete graph on $n$ vertices. To avoid these cases, in this paper, we deal with Johnson graphs $J(n,k)$ where $n\geq k+2 \ge 4$. 
\medskip

When we consider a graph $G$ that is a Kneser, bipartite Kneser, or a Johnson graph, the vertices will be denoted by italic capital letters, and \textit{the same notation will be used for a vertex of $G$ and the represented subset of $[n]$}. For a subset $\cS\subseteq V(G)$, we define the \textit{underlying hypergraph} $\cF(\cS)$ on the ground set $[n]$ that contains edges corresponding to the sets represented by the vertices in $\cS$. That is, if $G$ and $\cS \subseteq V(G)$ are given, 
$$E(\cF(\cS))=\{ S \subseteq [n]: S \in \cS \}.$$
If $G$ is a Kneser graph $KG(n,k)$ or a Johnson graph $J(n,k)$, then $\cF(\cS)$ is a $k$-graph for every nonempty vertex set $\cS$.
\medskip

\textit{Tur\' an-type problems} form a central area in extremal graph and hypergraph theory. 
Typically, in a hypergraph Tur\'an problem two integers $n$, $k$ and a $k$-graph $\cF$ are given. We want to  determine the maximum number of edges in a $k$-graph $\cH$ on an $n$-element vertex set such that $\cH$ contains no subhypergraph isomorphic to $\cF$. This maximum number is denoted by $\ex_k(n, \cF)$. If $k=2$ and $\cF$ is a graph, we omit the index $2$ and write simply $\ex(n, \cF)$. Remark that most of the hypergraph Tur\' an problems seem extremely difficult. Even the tight asymptotics for $\ex_k(n,\cF)$ as $n \to \infty$ might be hard to identify. For more results on the subject, we refer the reader to~\cite{keevash-2011} and \cite[Chapter 5]{gerbner-2019}.
\medskip

For three positive integers with $n \ge k \ge t$, an \textit{$(n, k, t)$ covering design} is a pair $(X, \cB)$, where $X$ is a set of $n$ elements
and $\cB$ is a collection of $k$-element subsets of $X$  (blocks) such that every $t$-element subset of $X$ is a subset of at least one
block from  $\cB$. The minimum number of blocks that may occur in an $(n,k,t)$ covering design is the \textit{covering number} $C(n,k,t)$. By double-countation, we obtain that ${k \choose t} C(n,k,t) \ge {n \choose t}$ is always true. Further, this inequality holds with equality if and only if a Steiner system $S(t,k,n)$ exists (see \cite{gordon}). This fact shows that determining the values $C(n,k,t)$ is a challenging problem in general. If $(X,\cB)$ is an $(n,k,t)$ covering design and $\cH=(X, \cB^c)$ is the $(n-k)$-uniform hypergraph containing the complements of the blocks $B \in \cB$, then no $t$-element subset of $X$ intersects every edge of $\cH$. Then, we have $\tau(\cH) > t$. As this implication also holds in the other direction, we may conclude that $C(n, n-k, t)$ is the minimum number of edges in a $k$-uniform hypergraph with $n$ vertices and transversal number $t+1$. This correlation was first remarked by D.B.~West  (see Chapter 7.1 in~\cite{Henning-Yeo}).

%%%%%%%%%%%%%%%%%%%%%%%%%%%%%%%%%%%%%%%%%%%
\subsection{Results and structure of the paper}

We start our study in Section~\ref{sec:1.5} by proving a lemma on the values of the covering number $C(n,k,t)$ under specific conditions.
 
In Section~\ref{sec:2}, the total mutual-visibility numbers of Kneser graphs $KG(n,k)$ are determined. While we prove explicit formulas in terms of $n$ and $k$ for the cases when $n \leq 3k-1$ or $2k^2 \leq n$, the formula for the middle range contains the constant $C(n, n-k, 2k-1)$. 

In Section~\ref{sec:3}, we determine the mutual-visibility number of $KG(n,k)$ for $n \ge 7k-5$ and estimate the value for the remaining cases.

In Section~\ref{sec:4}, the total mutual-visibility numbers of bipartite Kneser graphs $H(n,k)$ are determined. We also provide a lower bound for the mutual-visibility number $\mu(H(n,k))$.

Section~\ref{sec:5} is devoted to the study of Johnson graphs. The main theorem establishes equality between $\mut(J(n,k))$ and a hypergraph Tur\' an number for all nontrivial cases. The exact mutual-visibility number is proved for $J(n,2)$, and it is shown that $\mu(J(n,k))$ is sandwiched between two Tur\' an numbers in general.

In the last section, we answer a question from the recent manuscript~\cite{Cicerone-2024+} by determining the exact values for the mutual-visibility numbers of Kneser graphs $KG(n,2)$. Further, we give some remarks and formulate open problems related to the topic of this paper.

%%%%%%%%%%%%%%%%%%%%%%%%%%%%%%%%%%%%%%%%%%%%%%%
\section{A lemma on the covering number} \label{sec:1.5}

As a preparation for the proofs in Sections~\ref{sec:2} and \ref{sec:3}, here we mainly study the covering number.
First, we prove a short technical lemma.
\begin{lemma} \label{lem:binom}
If $n$ and $k$ are two positive integers and $k <n < 2k$, then
\begin{equation*} \label{eq:binom}
{n \choose k} > 2 \, {n-1 \choose k}.
\end{equation*}
%Further, if $n <2k$, then (\ref{eq:binom}) holds with strict inequality.
\end{lemma}
\begin{proof}
The condition $2k > n $ implies $\frac{n}{n-k} > 2$ and hence,
$$ {n \choose k} = \frac{n}{n-k}\, {n-1 \choose k} > 2\, {n-1 \choose k}
$$
holds. %If $2k > n $, then we have $\frac{n}{n-k} > 2$ that yields strict inequality in (\ref{eq:binom}). 
\end{proof}

Now, we concentrate on the covering numbers given in the form $C(n, n-k, 2k-1)$. To avoid trivial or non-defined cases, we will assume that $n \ge 3k$ and $k \ge 2$. Under these conditions, to simplify our notation, we set $$C^*(n,k)=C(n, n-k, 2k-1)$$ that is the minimum number of edges in a $k$-uniform hypergraph of order $n$ with transversal number $2k$.

\begin{lemma} \label{lem:C*}
Let $n$ and $k$ be integers.
\begin{itemize}
\item[$(i)$] If $k \ge 2$ and $n \ge 2k^2$, then $C^*(n,k)= 2k$.
\item[$(ii)$] If $k \ge 3$ and $n \ge 7k-5$, then $C^*(n,k) \le 2\, {2k-3 \choose k} +6.$
\item[$(iii)$] If $k \ge 2$ and $n \ge 2k^2+k$, then $C(n,n-k, 2k)= 2k+1$.
\end{itemize}

\end{lemma}
\begin{proof}
To prove $(i)$, we note that having $\tau = 2k$ needs at least $2k$ edges. Thus, $C^*(n,k) \ge 2k$. Moreover, the condition $n \ge 2k^2$ allows us to construct a $k$-uniform hypergraph on $n$ vertices with $2k$ pairwise disjoint edges. As its transversal number equals $2k$, the minimum number of edges in such a hypergraph is at most $2k$. Therefore, $C^*(n,k)=2k$.
\medskip

To prove $(ii)$ we consider a hypergraph $\cH_{n,k}$ built from the following types of components.
\begin{itemize}
\item A $k$-uniform generalized triangle $\cH_1^{(k)}$ has $n_1^{(k)}= k + \lceil \frac{k}{2} \rceil$ vertices partitioned into three sets $V_1,  V_2, V_3$  such that $|V_1|= \lfloor \frac{k}{2}\rfloor $ and $|V_2|=|V_3| =\lceil \frac{k}{2}\rceil $. If $k$ is even, we simply put
$$E(\cH_1^{(k)})= \{V_1 \cup V_2,  V_1 \cup V_3, V_2 \cup V_3 \}.$$
If $k$ is odd, we specify a vertex $x_3 \in V_3$ and replace the edge $V_2 \cup V_3$ with the $k$-edge $V_2 \cup (V_3 \setminus \{x_3\})$ in the definition of $E(\cH_1^{(k)})$. Then, $\cH_1^{(k)}$ has $m_1^{(k)}= 3$ edges and $\tau(\cH_1^{(k)})=2$ for all $k \ge 2$.
\item A complete $k$ uniform hypergraph $\cK_{2k-3}^{(k)}$ is defined on $2k-3$ vertices such that all $k$-element subsets of the vertex set are hyperedges in $\cK_{2k-3}^{(k)}$. Then, the size of this hypergraph is $m_2^{(k)}= {2k-3 \choose  k}$ and 
$$\tau(\cK_{2k-3}^{(k)})=(2k-3)-(k-1)=k-2.$$
\end{itemize}  
Now, to define the hypergraph $\cH_{n,k}$, take two disjoint copies of $\cH_1^{(k)}$ and two disjoint copies of $\cK_{2k-3}^{(k)}$. They together cover
$$2 \left(k + \left\lceil \frac{k}{2} \right\rceil \right) + 2(2k-3) \leq 7k-5 
 $$
vertices. The remaining vertices (if any) remain isolates. The transversal number of $\cH_{n,k}$ is  
$$\tau(\cH_{n,k}) = 2\cdot 2 + 2 (k-2)=2k.$$
By definition of the covering number, we may conclude that
$$C^*(n,k) \le |E(\cH_{n,k}  )| =2\, {2k-3 \choose k} +6.$$
%\medskip

The proof for part $(iii)$ is similar to that for $(i)$. Since $C(n, n-k, 2k)$ is the minimum number of edges in a $k$-uniform hypergraph on $n$ vertices having transversal number at least $2k+1$, at least $2k+1$ edges are needed. On the other hand, $n \ge 2k^2+k$ vertices are enough to take $2k+1$ pairwise vertex-disjoint $k$-edges and achieve $\tau \ge 2k+1$.
\end{proof}

%%%%%%%%%%%%%%%%%%%%%%%%%%%%%%%%%%%%%%%%%%%%%%
\section{Total mutual-visibility in Kneser graphs} \label{sec:2}

The main result of this section is Theorem~\ref{thm:mut-Kneser} that determines the total mutual-visibility number of Kneser graphs for all nontrivial cases. In its proof, we will use the following lemma. 

 \begin{lemma} \label{thm:mut-transv}
 Let $\cX$ be a set of vertices of the Kneser graph $KG(n,k)$ and $\overline{\cX}= V(KG(n,k)) \setminus \cX$. For every  $n \geq 3k-1$, the set $\cX$ is a total mutual-visibility set in $KG(n,k)$ if and only if $\tau(\cF(\overline{\cX}))\geq 2k$. 
 \end{lemma}
\begin{proof}
If $n \ge 3k-1$, then  $\diam(KG(n,k))=2$ (see~\cite{valencia-2005}).	
Suppose first that $\tau(\cF(\overline{\cX}))\geq 2k$ and consider two arbitrary vertices $A,B$ from the Kneser graph $G=KG(n,k)$. If $A$ and $B$ are adjacent, they are $\cX$-visible. Otherwise $\dist(A,B) = 2$, and the $k$-sets $A$ and $B$ are intersecting. Thus $|A\cup B| \leq 2k-1 <\tau(\cF(\overline{\cX})) $, and the set $A \cup B$ cannot intersect all  edges of $\cF(\overline\cX)$. Let us choose such an edge $Y$ with $Y\cap (A \cup B) = \emptyset$ from $\cF(\overline\cX)$. The corresponding vertex $Y$ belongs to $\overline{\cX}$ in $G$. As we have $Y\cap A = Y \cap B = \emptyset$, the vertex $Y$ is a common neighbor of $A$ and $B$. We conclude that $A$ and $B$ are $\cX$-visible, and hence, $\cX$ is a total mutual-visibility set in $G$. 

To prove the other direction, we suppose that $\tau(\cF(\overline{\cX}))\leq 2k-1$. Let us choose $2k-1$ elements, say $a_1,\dots, a_{k-1}, b_1,\dots, b_{k-1}, c,$ from the ground set $[n]$ such that they together intersect all edges of $\cF(\overline{\cX})$. Consider the two nonadjacent vertices of $G$ corresponding to the sets $A=\{a_1,\dots, a_{k-1},c\}$ and $B=\{b_1,\dots, b_{k-1},c\}$. Since $A \cup B$ intersects every $k$-edge from $\cF(\overline{\cX})$, the vertices  $A$ and $B$ in $G$ have no common neighbor from $\overline{\cX}$. As $\dist(A,B) = 2$ in $G$, we may conclude that $A$ and $B$ are not $\cX$-visible and, therefore $\cX$ is not a total mutual-visibility set.
\end{proof}

%For three integers $n,s,t$ with $n \ge s \ge t$, the constant $C(n,s,t)$ denotes the minimum number of $s$-element subsets of $[n]$ needed to cover all $t$-element subsets of $[n]$, where `covering' means that each $t$-set is contained in at least one selected $s$-set \cite{Henning-Yeo}. An equivalent formulation, due to West (see Chapter 7.1 in~\cite{Henning-Yeo}), is that $C(n, n-k, t)$ is the minimum number of edges in a $k$-uniform hypergraph on $n$ vertices with transversal number $t+1$. To simplify our notation, we define $C^*(n,k, 2k)= C(n, n-k, 2k-1)$ to denote the minimum size of a $k$-graph on $n$ vertices with $\tau=2k$. 

%First, we prove two statements on the value of $C^*(n,k,2k)$. ***

\begin{theorem} \label{thm:mut-Kneser}
	If $n$ and $k$ are integers such that $n\geq 2k+1$ and $k \ge 2$, then 
	\[\mut(KG(n,k)) = \left\{
	\begin{array}{ll}
		0 & \emph{if }2k+1 \leq n \leq 3k-1, \\
		{n \choose k} - C^*(n, k) & \emph{if } 3k\leq n \leq 2k^2-1,  \\
		{n \choose k} - 2k & \emph{if } 2k^2\leq n.  \\

	\end{array} 
	\right.\]
\end{theorem}

\begin{proof}
First, suppose that $2k+1 \leq n \leq 3k-1$ and a total mutual-visibility set $\cX$ in $G=KG(n,k)$ contains at least one vertex $A$. Let $\overline{A}$ denote the $(n-k)$-element set $[n] \setminus A$. As $k+1 \leq |\overline{A}| \leq 2k-1$, we can always find two $k$-element subsets $B$ and $C$ such that $B\cap C \neq \emptyset$ and $B\cup C=\overline{A}$. The corresponding vertices $B$ and $C$ are nonadjacent in $G$, and only vertex $A$ is adjacent to both. Consequently, $BAC$ is the unique shortest $B,C$-path in $G$. Since $A \in \cX$, the vertices $B$ and $C$ are not $\cX$-visible, and $\cX$ is not a total mutual visibility set in $G$. This contradiction proves $\mut(G)=0$ for the range $2k+1 \leq n \leq 3k-1$.
\medskip

Now, suppose that $n \ge 3k$. By Lemma~\ref{thm:mut-transv},  $\tau \ge 2k$ is exactly the property expected from the $k$-graph $\cF(\overline{\cX})$ to ensure that $\cX$ is a total mutual-visibility set. Thus $C^*(n,k)$ is the minimum number of $k$-edges in $\cF(\overline{\cX})$ that is the minimum number of vertices in $\overline{\cX}$ when $\cX$ is a total mutual-visibility set in $G$. Therefore $\mut(G)={n \choose k} - C^*(n,k)$, if $n \ge 3k$.

\medskip

 As $2k^2 \ge 3k$, the formula $\mut(G)={n \choose k} - C^*(n,k)$ remains valid  under the condition $n \ge 2k^2$. Moreover, in this case  we have $ C^*(n,k)=2k$ according to Lemma~\ref{lem:C*} (i). This finishes the proof of the theorem.
\end{proof}

\section{Mutual-visibility number of Kneser graphs} \label{sec:3}

Here we study the mutual-visibility number of Kneser graphs. Theorem~\ref{thm:Mu-Kneser} shows that $\mu (KG(n,k))$ equals $\mut KG(n,k)$ if $n \ge 7k-5$, while Proposition~\ref{prop:Kneser} states a lower bound for the range $2.5k \le n \le 7k-4$.

\begin{theorem} \label{thm:Mu-Kneser}
If $n$ and $k$ are two integers such that  $n \ge 7k-5$ and $k \ge 2$, then $$\mu(KG(n,k))={n \choose k} -C^*(n,k).$$
In particular, if $n \ge 2k^2$ and $k \ge 2$, then $$\mu(KG(n,k))={n \choose k} -2k.$$
\end{theorem}
\begin{proof}
	Since $\mu(KG(n,k)) \geq \mut (KG(n,k)) $ holds by definitions, Theorem~\ref{thm:mut-Kneser} directly implies $\mu(KG(n,k)) \geq {n \choose k} - C^*(n,k)$ when $n \ge 7k-5 > 3k$. Therefore, it is enough to prove that $ \mu(KG(n,k))  \le {n \choose k} - C^*(n,k)$ when $n \ge 7k-5$. 
	\medskip
	
	Let $k$ and $n$ be integers so that $n \ge 7k-5$ and denote the graph $KG(n,k)$ by $G$. Recall that $\diam(G)=2$ under the given conditions. Suppose for a contradiction that $\cX $ is a mutual-visibility set in $G$ and $|\overline{\cX}| \leq C^*(n,k)-1$. It follows that $\tau(\cF(\overline{\cX})) \le 2k-1$. By Lemma~\ref{lem:C*} (ii), the number of $k$-edges in $\cF(\overline{\cX})$, that is denoted by $m$ and equals $|\overline{\cX}|$, satisfies 
	$$m \le 2\, {2k-3 \choose k} +5.$$
	 Our aim is to find two $k$-element subsets $A$ and $B$ of $[n]$ such that 
	\begin{itemize}
	\item[$(\star)$] $A,B \notin E(\cF (\overline{\cX}))$, $A$ intersects $B$, and $A \cup B$ is a transversal in $\cF (\overline{\cX})$.
	\end{itemize}
	 Then, in the graph $G$, the vertices $A$ and $B$ belong to $\cX$ and have no common neighbor in $\overline{\cX}$. As $A$ and $B$ are nonadjacent and $\diam(G)=2$, there is no shortest $A,B$-path to ensure the mutual-visibility. This will give the contradiction and proves that $|\overline{\cX}| \ge C^*(n,k)$ when $\cX$ is a mutual-visibility set in $G$. 
		We consider three cases in the continuation.
	\paragraph{Case 1.} $k\ge 3$ and $\tau(\cF(\overline{\cX})) =2k-1$.\\
	Let $T$ be a minimum transversal in $\cF(\overline{\cX})$ and denote by $\cT$ the subhypergraph induced by $T$ in $\cF(\overline{\cX})$ . Then $|T|=2k-1$ and every vertex $u \in T$ is incident to a `private' edge $e_u$ from $\cF(\overline{\cX})$ such that the intersection $e_u \cap T$ contains only the vertex $u$. For the number of hyperedges in $\cT$, which is denoted by $m_\cT$, we therefore have
	$$m_\cT \le m-(2k-1) \le m-5 \le 2\, {2k-3 \choose k}.$$
	To estimate the average vertex degree $\overline{d}_\cT$ in the induced subhypergraph $\cT$, we apply Lemma~\ref{lem:binom}:
	\begin{eqnarray} \label{eq:average}
	\nonumber &\overline{d}_\cT = \displaystyle \frac{k\, m_\cT}{2k-1} &\leq 2\, {2k-3 \choose k}\, \frac{k}{2k-1}\\
	\nonumber & &< \frac{1}{2} \, {2k-1 \choose k} \, \frac{k}{2k-1}\\
	& & = \frac{1}{2} \, {2k-2 \choose k-1}.
	\end{eqnarray} 
	The upper bound implies that there exists a vertex $z\in T$ being incident to less than $\frac{1}{2} \, {2k-2 \choose k-1}$ hyperedges in $\cT$. Consider now the pairs $\{A', B'\}$ where $A'$ and $B'$ are disjoint $(k-1)$-element subsets of $T \setminus \{z\}$. Every $(k-1)$-element subset appers in exactly one such pair and there are $ \frac{1}{2} {2k-2 \choose k-1}$ pairs. By the pigeonhole principle, there exists a pair $\{A', B'\}$ such that neither $A=A'\cup \{z\}$ nor $B=B'\cup \{z\}$ is an edge in $\cT$. Consequently, the $k$-sets $A$ and $B$ are not edges in $\cF(\overline{\cX})$, they intersect in $z$, and the union $T=A \cup B$ is a transversal in $\cF(\overline{\cX})$. We may infer that $A$ and $B$ satisfy $(\star)$ and the corresponding vertices $A,B \in \cX$ are not $\cX$-visible in $G$. This contradiction proves $ \mu(KG(n,k))  \leq {n \choose k} - C^*(n,k)$.
	\paragraph{Case 2.} $k\ge 3$ and $\tau(\cF(\overline{\cX})) \le 2k-2$.\\
	First, suppose that there is an isolated vertex $z$ in $\cF(\overline{\cX})$. 	Then there exists a (not necessarily minimum) transversal $T$ in $\cF(\overline{\cX})$ such that $|T|=2k-2$ and $z \notin T$. For an arbitrary $(k-1)$-element subset $A'$ of $T$ and for the complement $B'= T \setminus A'$, the sets $A= A' \cup \{z\}$ and $B= B' \cup \{z\}$ satisfy the three properties in $(\star)$ and give the required contradiction.\\
	In the other case, there is no isolated vertex in $\cF(\overline{\cX}) $. Let $T$ be an arbitrary $(2k-1)$-element transversal of this hypergraph and denote by $m_\cT$ the number of hyperedges in the subhypergraph $\cT$ induced by $T$.  For the 
	$$n-|T| \ge (7k-5)-(2k-1)=5k-4 >5(k-1) $$
	 non-isolated vertices outside $T$, we need at least $6$ hyperedges to cover them all in $\cF(\overline{\cX})$. It implies $m_\cT \le m-6 < 2 {2k-3 \choose k}$. In the continuation, we can proceed as in Case~1; that is,  we estimate the average degree in $\cT$ according to $(\ref{eq:average}) $ and choose a vertex $z \in T$ of minimum degree. Considering the pairs of disjoint $(k-1)$-element subsets of $T\setminus \{z\}$, we may conclude that there exist $k$-sets $A$ and $B$ satisfying $(\star)$. The contradiction proves $ \mu(KG(n,k))  \leq {n \choose k} - C^*(n,k)$. 
	\paragraph{Case 3.} $k= 2$.\\
	 If $k =2$, the condition $n \ge 7k-5$ gives $n \ge 9$ and implies $n > 2k^2=8$. Hence, by Lemma~\ref{lem:C*} (i), $C^*(n,2)=4$ also holds. By our supposition, $\cX$ is a mutual visibility set in $G$ and $\cF(\overline{\cX})$ contains $m \le C^*(n,2)-1 =3$ edges. As $n \ge 9$, there exist isolated vertices. If $\tau(\cF(\overline{\cX})) =3$, then we have three vertex-disjoint hyperedges in $\cF(\overline{\cX}) $ and, for a minimum transversal $T$, it is easy to choose two 2-element sets $A \subset T$ and $B \subset T$ satisfying $(\star)$. If $\tau(\cF(\overline{\cX})) \le 2$, we take a $2$-element transversal $\{a,b\}$ and a third vertex $c$ that is isolated in $\cF(\overline{\cX})$. The sets $A= \{a,c\}$, $B=\{b,c\}$ satisfy $(\star)$. This gives the desired contradiction and proves the formula for $k=2$.
	 \bigskip
	 
	 If $k \ge 3$ and $n \ge 2k^2$, then $n \ge 7k-5$ also holds. Thus the formula $ \mu(KG(n,k)) = {n \choose k} - C^*(n,k)$ and Lemma~\ref{lem:C*} (i) directly imply $ \mu(KG(n,k)) = {n \choose k} - 2k$. The same is true if $k=2$ and $n \ge 9$. The only remaining case is the graph $KG(8,2)$. Here, a mutual-visibility set $\cX$ with more than ${8 \choose 2}- 4$ vertices defines the $2$-uniform underlying hypergraph $\cF(\overline{\cX})$ with $n =8$ vertices and at most three edges. As in Case~3 of the proof, it is easy to show two non-edges in  $\cF(\overline{\cX})$ that satisfy $(\star)$ and prove that $\cX$ is not a mutual visibility set in $KG(8,2)$.
\end{proof}

The mutual-visibility notion is related to the general position problem in graphs. A subset $S$ of vertices in a connected graph $G$ is a \textit{general position set} if no triple of vertices from $S$ lie on a common shortest path in $G$. Equivalently, for every three different vertices $u,v,z \in S$, the strict inequality $\dist(u,v)+\dist(v,z) > \dist(u,z)$ holds. The general position problem \cite{ghorbani2019, manuel} is to find a largest general position set of $G$, the order of such a set is the \textit{general position number} $\gp(G)$. By definitions, every general position set is a mutual visibility set and therefore, $\mu(G) \geq \gp(G)$ holds for every connected graph $G$. 
\begin{proposition} \label{prop:Kneser}
If $n$ and $k$ are integers satisfying $2.5 k-0.5 \leq n \leq 7k-5$ and $k \ge 2$, then
$$\mu(KG(n,k)) \ge {n-1 \choose k-1}.$$
\end{proposition}
\begin{proof}
The results from~\cite{ghorbani2019} and \cite{Patkos2020} yield that $\gp(KG(n,k))= {n-1 \choose k-1}$ holds when $n \ge 2.5 k-0.5$. Then, the relation $\mu(KG(n,k)) \ge \gp(KG(n,k))$ directly implies the statement.
\end{proof}

%%%%%%%%%%%%%%%%%%%%%%%%%%%%%%%%%%%%%%%%%%%%%%%%%%%%%%
\section{Bipartite Kneser graphs} \label{sec:4}
In this section we study the total mutual-visibility number and mutual-visibility number of bipartite Kneser graphs. Since a $k$-element subset $S_1 \subseteq [n]$ is a subset of an $(n-k)$-element subset $S_2 \subseteq [n]$ if and only if $[n]\setminus  S_2 \subseteq [n] \setminus S_1$, bipartite Kneser graphs satisfy the following property:
\begin{observation} \label{obs:sym} 
Let $G$ be the bipartite Kneser graph $H(n,k)$ such that the vertices in the partite classes $\cP_1$ and $\cP_2$ represent the $k$-element and $(n-k)$-element subsets of $[n]$, respectively. Let $\phi$ be the function that assigns to each vertex $S \in V(H(n,k))$ the vertex $S^c$ which represents the complement set $S^c=[n] \setminus S$. Then, $\phi$ is an automorhism of $H(n,k)$ that maps $\cP_1$ into $\cP_2$. 
\end{observation} 
It follows from the observation that any property proved for $\cP_1$ holds for $\cP_2$, and vice versa. Remark that bipartite Kneser graphs satisfy the much stronger property of symmetry i.e., they are vertex- and edge-transitive graphs~\cite{simpson}.
\medskip

In the main result of this section, we will use again the constant $C(n,k,t)$ that is the minimum size of a covering design with the given parameters.

\begin{theorem} \label{thm:bipart}
	If $n$ and $k$ are integers such that $n\geq 2k+1$, then 
\[\mut(H(n,k)) = \left\{
\begin{array}{ll}
	0 & \emph{if }2k+1 \leq n \leq 3k, \\
  2\,	{n \choose k} - 2\, C(n, n-k, 2k) & \emph{if } 3k+1\leq n < 2k^2+k,  \\	
  2\,	{n \choose k} - 4k-2 & \emph{if } 2k^2+k\leq n.  
\end{array} 
\right.\]
\end{theorem}
\begin{proof}
Consider a bipartite Kneser graph $H(n,k)$ with parameters satisfying $n \ge 2k+1$, and let $\cP_1$ and $\cP_2$ denote the two partite classes of $H(n,k)$ such that $\cP_1$ consists of the vertices representing the $k$-element subsets of $[n]$.
Let $\cX$ be a total mutual visibility set in $H(n,k)$.
\medskip

We first assume that $n \le 3k$ and show that $\cX \cap \cP_1 =\cX \cap \cP_2 = \emptyset$. Note that $n \le 3k$ implies $k+ 2(n-2k) \leq n$ and then, for every vertex $A \in \cP_1$, we can choose two $(n-k)$-element sets $B_1 $ and $B_2$ such that $B_1 \cap B_2=A$. Vertex $A$ is clearly the only common neighbor of $B_1$ and $B_2$ in $H(n,k)$. The $\cX$-visibility of $B_1$ and $B_2$ therefore implies $A \notin \cX$. By Observation~\ref{obs:sym}, we infer that the same is true for the vertices in $\cP_2$ and therefore,  $\cX=\emptyset$. It proves  $\mut(H(n,k))=0$ for the case $n \leq 3k$.
%Similarly, consider now an arbitrary vertex $C \in \cP_2$. From the corresponding $(n-k)$-element subset $C$ of $[n]$, we choose two $k$-element subsets $D_1$ and $D_2$ such that $D_1 \cup D_2= C$. It can be done as $k <n-k \le 2k$. Again, we find that $C$ is the only common neighbor of $D_1$ and $D_2$ and therefore, $C \notin \cX$. It proves that $\cX=\emptyset$ and $\mut(H(n,k))=0$ if $n \leq 3k$.
\medskip

Assume now that $n \geq 3k+1$ and show that the family $\cP_2 \cap \overline{\cX}$ of $(n-k)$-element sets covers all $(2k)$-element subsets of $[n]$. Suppose to the contrary that no $B \in \cP_2 \cap \overline{\cX}$ covers the $(2k)$-element set $S \subseteq [n]$. Split $S$ into two disjoint $k$-sets $A_1$ and $A_2$. Since $n-k > 2k$, there exist $(n-k)$-element subsets of $[n]$ that cover both $A_1$ and $A_2$. However, by our supposition, none of the vertices corresponding to these covering sets belong to $\cP_2 \cap \overline{\cX}$. Equivalently, in $H(n,k)$, $\dist(A_1, A_2)=2$ but none of the common neighbors belong to $\overline{\cX}$. This contradicts the $\cX$-visibility of $A_1$ and $A_2$. We infer that the $(n-k)$-sets in $\cP_2 \cap \overline{\cX}$ cover all $(2k)$-subsets of $[n]$ and therefore, $|\cP_2 \cap \overline{\cX}|\ge C(n,n-k, 2k)$. By Observation~\ref{obs:sym}, $|\cP_1 \cap \overline{\cX}|\ge C(n,n-k, 2k)$ follows and we may conclude
$$\mut(H(n,k)) \le   2\,	{n \choose k} - 2\, C(n, n-k, 2k).$$

To prove the other direction, let $\cB$ be a family of $(n-k)$-element subsets of $[n]$ that cover all $(2k)$-element subsets. We also suppose that $\cB$ is minimum that is, $|\cB|= C(n, n-k, 2k)$ and use the same notation $\cB$ for the set of corresponding vertices in $\cP_2$. Moreover, let 
$$\cA = \{A \subseteq [n]: [n]\setminus A  \in \cB\}= \phi(\cB),$$
where $\phi$ is the automorphism from Observation~\ref{obs:sym}.
%By the complementation, each $(n-2k)$-element subset of $[n]$ has a $k$-element subset $A$ which belongs to the family $\cA$.  Clearly, $|\cA| =|\cB|$. 
We are going to prove that the vertex set $\cX=\overline{\cA \cup \cB}$ forms a total mutual-visibility set in $H(n,k)$. 

Take two vertices $A_1$ and $A_2$ from the partite class $\cP_1$ of $H(n.k)$. Since $|A_1 \cup A_2| \leq 2k$, the union is covered by an $(n-k)$-element set  $B \in \cB$. The corresponding vertex  $B \in \overline{\cX}$ is a common neighbor of $A_1$ and $A_2$ and consequently $A_1$ and $A_2$ are $\cX$-visible. By Observation~\ref{obs:sym} and by the symmetry of our definitions for $\cA$ and $\cB$, the same is true for any two vertices $B_1, B_2 \in \cP_2$.

The last case to check is when $A \in \cP_1$ and $B \in \cP_2$. If they are adjacent, the $\cX$-visibility follows. If they are nonadjacent vertices, the distance is at least $3$ (in fact, it is exactly $3$). We may choose an arbitrary vertex $Y$ from $\cP_1$ that is different from $A$. As we have already seen, there exists a vertex $B_1 \in \cP_2 \cap \overline{\cX}$ which is a common neighbor of $A$ and $Y$. Similarly, there is a vertex $A_1$ in $  \cP_1 \cap \overline{\cX}$ which is a common neighbor of $B$ and $B_1$. We may infer that $AB_1A_1B$ is a shortest $A,B$-path and both internal vertices belong to $\overline{\cX}$. Therefore, $A$ and $B$ are $\overline{\cX}$-visible.  It finishes the proof for $\cX= \overline{\cA \cup \cB}$ being a total mutual-visibility set in $H(n,k)$. We may conclude that
$$\mut(H(n,k)) \ge |\cX| = 2\, {n \choose k}-|\cA \cup \cB| = 2\,	{n \choose k} - 2\, C(n, n-k, 2k).$$
This proves the desired formula for $n \ge 3k+1$.
\medskip

If $n \ge 2k^2+k$, the equality $\mut(H(n,k)) =2\,	{n \choose k} - 2\, C(n, n-k, 2k)$ still remains valid, and together with Lemma~\ref{lem:C*} $(iii)$ they yield the formula stated for the last case.
\end{proof}

%\begin{corollary}
%If $k$ is a fixed integer with $k\ge 2$, then
%$$\mut(H(n,k)) = (2-o(1))  {n \choose k}.$$
%\end{corollary}

Concerning the mutual-visibility number of bipartite Kneser graphs, we propose two simple lower bounds.
\begin{proposition} \label{prop:bipK}
If $n$ and $k$ are integers such that $n\geq 3k+1$, then
$$\mu(H(n,k)) \ge \max \left\{ {n \choose k}, \,  2\,	{n \choose k} - 2\, C(n, n-k, 2k) \right\}.$$
\end{proposition}
\begin{proof}
By definition,  $\mu(G) \ge \mut(G)$ holds for every graph $G$. Theorem~\ref{thm:bipart} then implies $\mu(H(n,k)) \ge 2\,	{n \choose k} - 2\, C(n, n-k, 2k)$. 
The inequality $\mu(H(n,k)) \ge {n \choose k}$ follows from the fact that $\cX= \cP_1$ is a mutual-visibility set if $n \ge 3k+1$. Indeed, as $\dist(A_1,A_2)=2$ holds for every two vertices $A_1, A_2$ from $\cP_1= \cX$, each shortest path between them contains only one internal vertex, and this vertex is from $\cP_2=\overline{\cX}$. 
%We have seen in the proof of Theorem~\ref{thm:bipart} that the diameter of $H(n,k)$ equals $3$ if $n \ge 3k+1$. Then it follows that (***cite)  $\mut(H(n,k)) \ge \alpha(H(n,k))$. Since $H(n,k)$ is a regular bipartite graph, it admits a perfect matching and the independence number equals the size of the partite classes. This proves $\mu(H(n,k))\ge {n \choose k}$.
\end{proof}

%%%%%%%%%%%%%%%%%%%%%%%%%%%%%%%%%%%%%%%%%%%%%
\section{Johnson graphs} \label{sec:5}

In this section, we determine the total mutual-visibility number and estimate the mutual visibility number of Johnson graphs in terms of graph and hypergraph Tur\'{a}n numbers. Before stating the main results, we define the \textit{$k$-uniform suspensions}  of the $4$-cycle $C_4$ and the complete graph $K_4$ for every $k \ge 2$. The $k$-uniform suspensions are denoted by $\cC_4^{k+}$ and $\cK_4^{k+}$ respectively. The $(k+2)$-element vertex set for both of them is given as $Y \cup \{z_1,z_2,z_3,z_4\}$, and the two edge sets are the following:\footnote{Note that $\cC_4^{k+}$ corresponds to the complete $k$-partite $k$-graph $\cK_{1,\dots, 1, 2,2}^{(k)}$.}
\begin{eqnarray} \label{eq:1}
	E(\cC_4^{k+})&=&\{Y \cup \{z_i,z_{i+1}\}\}: i \in [3]\} \cup \{Y \cup \{z_4,z_1\}\}\\
	\label{eq:2}
	E(\cK_4^{k+})&=&\{Y \cup \{z_i,z_{j}\}\}: i,j \in [4], \enskip i< j\}.
\end{eqnarray}

The following tight asymptotic result was proved by Mubayi:
\begin{theorem}[\cite{mubayi-2002}] \label{mubayi} 
	For every $k \ge 2$,
	$$\ex_k(n, \cC_4^{k+})=(1+o(1)) \frac{1}{k!} \, n^{k-0.5}.
	$$
\end{theorem}
The main theorem of this section establishes a direct connection between total mutual-visibility numbers of Johnson graphs and hypergraph Tur\' an numbers $\ex_k(n, \cC_4^{k+})$. For $k=2$, the Johnson graph $J(n,2)$ is the line graph of the complete graph $K_n$. Therefore, Theorem~3.7 in~\cite{Cicerone-2024+} has already established the equality $\mut(J(n,2))=\ex(n, C_4)$. Here we use a different method to approach the general problem.
\begin{theorem} \label{thm:mut-Johnson}
	It holds for every two integers $n$ and $k$ satisfying $n \ge k +2 $ and $k \ge 2$ that
	$$ \mut(J(n,k)) = \ex_k(n, \cC_4^{k+})
	$$
\end{theorem}
\begin{proof}
	Assume that $\cX$ is a total mutual visibility set of maximum cardinality in $J(n,k)$. We first prove that the $k$-uniform underlying hypergraph $\cF(\cX)$ is $\cC_4^{k+}$-free. Suppose to the contrary that $\cC_4^{k+}$ is a subhypergraph in $\cF(\cX)$. Naming the vertices in this subhypergraph as in (\ref{eq:1}), let us consider the $k$-element sets $A_1= Y \cup \{z_1, z_3\}$ and $A_2=Y \cup \{z_2, z_4\}$. The corresponding vertices $A_1$ and $A_2$ are at distance $k-|Y|= 2$ in $J(n,k)$. Further, $A_1$ and $A_2$ have exactly four common neighbors in $J(n,k)$ and these neighbors correspond to the four edges in the subhypergraph $\cC_4^{k+}$. By the supposition, the corresponding four vertices all belong to $\cX$ and therefore, $A_1$ and $A_2$ are not $\cX$-visible. This contradiction proves that $\cF(\cX)$ is $\cC_4^{k+}$-free and, since the size of $\cX$ equals the number of edges in $\cF(\cX)$, we have 
	$$\mut(J(n,k)) = |\cX| \leq \ex_k(n, \cC_4^{k+} ).$$
	
	To prove the other direction, we suppose that $\cF(\cX)$ is $\cC_4^{k+}$-free and prove that every two vertices $A$ and $B$ are $\cX$-visible in $J(n,k)$. We proceed by induction on the distance of $A$ and $B$. If $\dist(A,B)=1$, then the statement clearly holds. Suppose that $A \cap B =D$ holds for the corresponding $k$-sets and then $\dist (A,B)=k-|D| \ge 2$. Let us choose four different vertices $z_1, z_2 \in A \setminus D$ and $w_1, w_2 \in B \setminus D$. Since $\cF(\cX)$ is $\cC_4^{k+}$-free, there are two indices $i$ and $j$ such that $A' = A \setminus \{z_i\} \cup \{w_j\}$ is not an edge in $\cF(\cX)$. Then, for the corresponding vertex, $A' \notin \cX$. As $\dist(A', B)= \dist(A,B)-1$, we can now apply the hypothesis for $A'$ and $B$. As there is a shortest path between $A'$ and $B$ such that all internal vertices belong to $\overline{\cX}$ and $A'$ itself is from $\overline{\cX}$, we can construct a shortest path between $A$ and $B$ that contains no internal vertex from $\cX$. This proves that $A$ and $B$ are $\cX$-visible and we can conclude that $\cX$ is a total mutual-visibility set if $\cF(\cX)$ is $\cC_4^{k+}$-free. 
	
	By choosing a $k$-uniform  $\cC_4^{k+}$-free hypergraph with maximum number of edges, the corresponding vertex set $\cX$ will be a total mutual-visibility set in $J(n,k)$. This proves $\mut(J(n,k)) \ge \ex_k(n, \cC_4^{k+})$ and finishes the proof of the theorem.	
\end{proof}
By Theorems~\ref{mubayi} and~\ref{thm:mut-Johnson}, we may conclude the following asymptotically tight result:
\begin{corollary} \label{cor:Johnson-asympt}
For every fixed $k \ge 2$,
$$\mut(J(n,k))=(1+o(1)) \frac{1}{k!} \, n^{k-0.5}.
$$
\end{corollary} 
\medskip

Finally, we turn to the problem of mutual-visibility in Johnson graphs. Here, we can give an exact formula for $k=2$ and $k=n-2$ (that was proved in~\cite{Cicerone-2024+} with a different approach), while establishing lower and upper bounds for the remaining cases.
\begin{theorem} \label{thm:MU-Johnson}
	Let $n$ and $k $ be integers. 
	\begin{itemize}
	\item[$(i)$] If $k \ge 2$ and $n \ge k +2 $, then
	\begin{equation} \label{eq:3}
	\ex_k(n, \cC_4^{k+}) \leq \mu(J(n,k)) \leq  \ex_k(n, \cK_4^{k+}).
	\end{equation}
	\item[$(ii)$]%~\cite{Cicerone-2024+} 
	If $n \ge 4$, then
	\begin{equation} \label{eq:4}
	\mu(J(n,2))=\mu(J(n,n-2))= \left\lfloor \frac{n^2}{3} \right\rfloor.
	\end{equation}
	\end{itemize}
	\end{theorem}
\begin{proof}
By definitions of the parameters, $\mu(G) \ge \mut(G)$ holds for every graph $G$. Theorem~\ref{thm:mut-Johnson} then directly implies 	$\ex_k(n, \cC_4^{k+}) \leq \mu(J(n,k))$. 

To show the upper bound in (\ref{eq:3}), we prove that the $k$-uniform underlying hypergraph $\cF(\cX)$ is $\cK_4^{k+}$-free whenever $\cX$ is a mutual visibility set in $J(n,k)$. Indeed, let us suppose that $\cF(\cX)$ contains a subhypergraph on the vertex set $Y \cup \{z_1, \dots, z_4\}$ that is isomorphic to $\cK_4^{k+}$. We name the vertices according to~(\ref{eq:2}). Then, by the definition of underlying hypergraph $\cF(\cX)$, the vertices of $J(n,k)$ that correspond to the $k$-sets $Y \cup \{z_1,z_2\}$ and $Y \cup \{z_3,z_4\}$ belong to $\cX$. Their distance equals $2$ and all the four common neighbors are from $\cX$ as well. Thus, the two vertices  $Y \cup \{z_1,z_2\}$ and $Y \cup \{z_3,z_4\}$ from $\cX$ are not $\cX$-visible, and $\cX$ is not a mutual visibility set in $J(n,k)$. We may conclude that $\cF(\cX)$ is $\cK_4^{k+}$-free for every mutual visibility set $\cX$ of $J(n,k)$ and consequently, $\mu(J(n,k)) \leq  \ex_k(n, \cK_4^{k+})$ is valid.
\medskip

To prove~(\ref{eq:4}), we first recall the well-known facts that $\ex(n, K_4)= \lfloor n^2 /3 \rfloor$ \cite{turan} and  that $J(n,2) \cong J(n,n-2)$ holds for every $n \ge 4$. It is enough then to prove the equality  $\mu(J(n,2))=\ex(n,K_4)$. As $\cK_4^{2+}$ is the graph $K_4$, part $(i)$ provides the upper bound $\mu(J(n,2)) \le \ex(n,K_4)$.

Consider a set $\cX \subseteq V(J(n,2))$ so that the $2$-uniform underlying hypergraph $\cF(\cX)$ is $K_4$-free. For every two vertices $A$ and $B$ from $\cX$, they are either adjacent and $\cX$-visible or $\dist(A,B)=2$. In the latter case, the corresponding $2$-element sets are disjoint, say $A=\{a_1,a_2\}$ and $B=\{b_1,b_2\}$. Since $\cF(\cX)$ is $K_4$-free, there is a $2$-element set $\{a_i,b_j\}$ with $i,j \in [2]$ that is not an edge in $\cF(\cX)$. Thus $C=\{a_i,b_j\}$ is in $\overline{\cX}$ and $A$, $B$ are $\cX$-visible via the path $ACB$. Consequently, a mutual-visibility set $\cX$ can be chosen in $J(n,2)$ such that $\cF(\cX)$ is an extremal $K_4$-free graph having $\ex(n, K_4)$ edges. We conclude $\ex(n,K_4) \le \mu(J(n,2))$ that finishes the proof of (\ref{eq:4}).
\end{proof}

%%%%%%%%%%%%%%%%%%%%%%%%%%%%%%%%%%%%%%%%%%%%%%
\section{Conclusion} \label{sec:concl}

In this paper, we determined the total mutual-visibility numbers for all Kneser graphs, bipartite Kneser graphs, and Johnson graphs. Our formulas gave easily computable values for some ranges, but we had to include invariants from extremal combinatorics in other cases. Computing these invariants (namely hypergraph Tur\'an numbers of suspensions and covering numbers from design theory) is considered an extremely challenging problem.
\medskip

We also studied the mutual-visibility numbers over these three graph classes and obtained exact results for all Kneser graphs $KG(n,k)$ with $n \ge 7k-5$ and $k \ge 2$ (see~Theorem~\ref{thm:Mu-Kneser}). For the remaining cases, we gave estimations (see Propositions~\ref{prop:Kneser}, \ref{prop:bipK}, and Theorem~\ref{thm:MU-Johnson}), the exact determination is posed here as an open problem.
\begin{problem}
Determine the mutual visibility number of the Kneser graph $KG(n,k)$ if $3 \le k$ and $2k+1 \le n \le 7k-6$.
\end{problem}
\begin{problem}
Determine the mutual visibility number of the bipartite Kneser graph $H(n,k)$ if $k \ge 2$ and $n \ge 2k+1$.
\end{problem}
\begin{problem}
	Determine the mutual visibility number of the Johnson graph $H(n,k)$ if $k \ge 3$ and $n \ge k+2$.
\end{problem}

For every graph $G$, the relation $\mu(G)\ge \mut(G)$ holds by definition. If $G$ satisfies $\mu(G)=\mut(G)$, we call it \textit{$(\mu, \mut)$-graph}. Our Theorems~\ref{thm:mut-Kneser} and \ref{thm:Mu-Kneser} imply that all Kneser graphs $KG(n,k)$ with $k \ge  2$ and $n \ge 7k-5$ are $(\mu,\mut)$-graphs. Further, by the same theorems, $K(8,2)$ is also a $(\mu,\mut)$-graph. 
\medskip

Cicerone \textit{et al.}\ \cite{Cicerone-2023+4} introduced further parameters on mutual-visibility. Given a graph $G$ and a vertex set $\cX$, we say that $\cX$ is a \textit{dual mutual-visibility set}, if every two vertices from $\cX$ and every two vertices from $\overline{\cX}$ are $\cX$-visible. The maximum size of a dual mutual-visibility set in $G$ is the \textit{dual mutual-visibility number} $\mud(G)$. Similarly, if $\cX$-visibility is required for every two vertices $u,v$ when at least one of them is from $\cX$, then $\cX$ is an \textit{outer mutual-visibility set} and the maximum size of such a set is denoted by $\muo(G)$. It holds by definitions that
\begin{equation} \label{eq:dual}
\mut(G) \le \mud(G) \leq \mu(G) \qquad \mbox{and} \qquad \mut(G) \le \muo(G) \leq \mu(G).
\end{equation}
It follows that the four parameters are equal for each $(\mu, \mut)$-graph.
\medskip

As a consequence of (\ref{eq:dual}), Theorems~\ref{thm:mut-Kneser} and \ref{thm:Mu-Kneser}, it is easy to determine the four mutual-visibility parameters for Kneser graphs $KG(n,2)$ with $n \ge 8$. By doing so, we answer a question of Cicerone \textit{et al.} posed in~\cite{Cicerone-2024+}.
\begin{proposition}
It holds for every integer $n$ with $n \ge 8$ that
$$ \mu(KG(n,2))= \mud(KG(n,2))= \muo(KG(n,2))= \mut(KG(n,2))= {n \choose 2} -4.
$$
\end{proposition}

%%%%%%%%%%%%%%%%%%%%%%%%%%%%%%%%%%%%%%%%%%%%%%%%%%%%%%%
\section*{Acknowledgements}

This research was initiated with the visit of Csilla Bujt\'as to Ege University in Izmir,
Turkey; the authors thank the financial support of TÜBITAK under the grant BIDEB 2221 (1059B212300041). The research of Csilla Bujt\' as was partially supported by the 
Slovenian Research and Innovation Agency (ARIS) under the grant P1-0297.

%%%%%%%%%%%%%%%%%%%%%%%%%%%%%%%%%%%%%%%%%%%%%%%%%%%%%


%% \begin{thebibliography}{00}


%% \bibitem[ ()]{}

%% \end{thebibliography}

\begin{thebibliography}{99}
		
	\bibitem{adhikary-2022}
	R.~Adhikary, K.~Bose, M.K.~Kundu, B.~Sau,
	Mutual visibility on grid by asynchronous luminous robots,
	Theoret.\ Comput.\ Sci.\ 922 (2022) 218--247.
	
	\bibitem{bresar2023lower}
	B.~Bre{\v{s}}ar, I.G.~Yero,
	Lower (total) mutual visibility in graphs
	arXiv:2307.02951 [math.CO] (6 July 2023).
	
	
	\bibitem{Bujtas-2023}
	Cs.~Bujt\' as, S.~Klav\v zar, J.~Tian,
	Total mutual-visibility in Hamming graphs,
		arXiv:2307.05168 [math.CO] (11 July 2023).

	
	\bibitem{Cicerone-2023+}
	S.~Cicerone, A.~Di Fonso, G.~Di Stefano, A.~Navarra,
	The geodesic mutual visibility problem for oblivious robots: the case of trees.
	In 24th International Conference on Distributed Computing and Networking (ICDCN).
	ACM, New York. (2023, January 4-7) (150--159). doi:10.1145/3571306.3571401.
	
	\bibitem{Cicerone-2023+3}
	S.~Cicerone, A.~{Di Fonso}, G.~{Di Stefano}, A.~{Navarra}, F.~Piselli,
	Mutual visibility in hypercube-like graphs,
	arXiv:2308.14443 [math.CO] (28 August 2023).
	
		\bibitem{Cicerone-2023+5}
	S.~Cicerone, G.~{Di Stefano},
	Mutual-visibility in distance-hereditary graphs: a linear-time algorithm,
	arXiv:2307.10661 [math.CO] (20 July 2023).
	
	\bibitem{Cicerone-2023+4}
	S.~Cicerone, G.~{Di Stefano}, L.~Dro{\v{z}}{\dj}ek, J.~Hed{\v{z}}et, S.~Klav\v{z}ar, I.G. Yero,
	Variety of mutual-visibility problems in graphs
	Theor.\ Comput.\ Sci.\ 974 (2023) 114096
	
	

	
	\bibitem{Cicerone-2023}
	S.~Cicerone, G.~{Di Stefano}, S.~Klav\v{z}ar,
	On the mutual-visibility in Cartesian products and in triangle-free graphs,
	Appl.\ Math.\ Comput.\ 438 (2023) 127619.
	
		
	\bibitem{Cicerone-2023+2}
	S.~Cicerone, G.~{Di Stefano}, S.~Klav\v{z}ar, I.G. Yero,
	Mutual-visibility in strong products of graphs via total mutual-visibility,
	arXiv:2210.07835 [math.CO] (14 Oct 2022).
	
	\bibitem{Cicerone-2024+}
	S.~Cicerone, G.~{Di Stefano}, S.~Klav\v{z}ar, I.G. Yero,
	Mutual-visibility problems on graphs of diameter two,
	arXiv:2401.02373 [math.CO] (4 Jan 2024).
	
		
			
	\bibitem{diluna-2017}
	G.A.~Di Luna, P.~Flocchini, S.G.~Chaudhuri, F.~Poloni, N.~Santoro, G.~Viglietta,
	Mutual visibility by luminous robots without collisions,
	Inf.\ Comput.\ 254 (2017) 392--418.
	
	\bibitem{distefano-2022}
	G.~{Di Stefano},
	Mutual visibility in graphs,
	Appl.\ Math. Comput.\ 419 (2022) 126850.
	
	\bibitem{gerbner-2019}
	D.~Gerbner, B.~Patk\'{o}s,
	Extremal Finite Set Theory,
	CRC Press, Boca Raton, FL, 2019.
	
	\bibitem{ghorbani2019}
	M.~Ghorbani, S.~Ghorbani,  H.R.~Maimani, M.~Momeni, F.R.~Mahid, G.~Rus
	The general position problem on Kneser graphs and on some graph operations
	Discuss.\ Math.\ Graph Theory 41 (2021) 1199--1213.  
	
	\bibitem{gordon}
	D.M.~Gordon, D.R.~Stinson, 
	Coverings. 
	In: Handbook of Combinatorial Designs, eds: C.J.~Colbourn, J.H.~Dinitz, Boca Raton FL: Chapman \& Hall/Taylor \& Francis, 2007.
	
		\bibitem{Henning-Yeo}
		M.A.~Henning, A.~Yeo,
		Transversals in linear uniform hypergraphs.
		Springer, Cham, 2020.
	
 \bibitem{keevash-2011}
		P.~Keevash,
		Hypergraph Tur\'{a}n problems. In: Surveys in Combinatorics, Cambridge University Press, 2011, pp.\ 83--139.
		
	\bibitem{korze-2023}
	    D.~Kor\v ze, A.~Vesel,
		Mutual-visibility sets in Cartesian products of
		paths and cycles,
		arXiv:2309.15201 [math.CO] (26 Sep 2023).
	
		
		\bibitem{kuziak-2023+}
	D.~Kuziak, J.A.~Rodr\'{\i}guez-Vel\'{a}zquez,
	Total mutual-visibility in graphs with emphasis on lexicographic and Cartesian products,
	 Bull.\ Malays.\ Math.\ Sci.\ Soc.\ 46, 197 (2023). 
	 
	 \bibitem{manuel}
	 P.~Manuel, S.~Klav\v{z}ar,
	 A general position problem in graph theory, 
	 Bull.\ Aust.\ Math.\ Soc.\ 98 (2018) 177--187.
	
	\bibitem{mubayi-2002}	
	D.~Mubayi,
	Some exact results and new asymptotics for hypergraph
	Tur\'{a}n numbers,
	Combin.\ Probab.\ Comput.\ 11 (2002) 299--309.
	
	\bibitem{Patkos2020}
	B.~Patk{\'o}s,
	On the general position problem on Kneser graphs
	ARS Math.\  Contemp.\  18 (2020)  273--280.
	
	\bibitem{simpson}
	J.E.~Simpson, 
	Hamiltonian Bipartite Graphs, 
	Proceedings of the Twenty-second Southeastern Conference on Combinatorics, Graph Theory, and Computing (1991) 97--110. 
	%85, pp. 97-110, 1991.
	
	\bibitem{poudel-2021}
	P.~Poudel, A.~Aljohani, G.~Sharma,
	Fault-tolerant complete visibility for asynchronous robots with lights under one-axis agreement,
	Theor.\ Comput.\ Sci.\ 850 (2021) 116--134.
	
	\bibitem{tian-2023+}
	J.~Tian, S.~Klav\v{z}ar,
	Graphs with total mutual-visibility number zero and total mutual-visibility in Cartesian products,
	Discuss.\ Math.\ Graph Theory (2023)  https://doi.org/10.7151/dmgt.2496.
	
	\bibitem{turan}
	P. Tur\'  an. 
	On an extremal problem in graph theory (in Hungarian).
	Matematikai \' es Fizikai Lapok, 48 (1941) 436--452.
	
	\bibitem{valencia-2005}
	M.~Valencia-Pabon, J.~Vera,
	On the diameter of Kneser graphs,
	Discrete Math.\ 305 (2005) 383--385.

	
	%\bibitem{W-93}
	
	\bibitem{west-2021}
	D.B.~West,
	Combinatorial Mathematics,
	Cambridge University Press, Cambridge, 2021.
	
	
\end{thebibliography}
\end{document}